\documentclass[12pt]{article}
\usepackage[hidelinks]{hyperref} 
\usepackage[utf8]{inputenc}
\usepackage{hyperref}
\usepackage{amsthm}
\usepackage{amsmath}
\usepackage{enumitem}
\usepackage{amssymb}
\usepackage{tikz}
\usepackage{tikz-cd}
\usetikzlibrary{positioning}
\usepackage{graphicx}
\usepackage{float}
\usepackage{cleveref}
\usepackage{geometry}
\usepackage{thm-restate}

\usepackage{MnSymbol}
\usepackage{todonotes}

\usepackage[all]{xy}

\title{On the computability of optimal Scott sentences}
\author{Rachael Alvir, Barbara Csima\thanks{The second author was partially supported by an NSERC Discovery Grant under Grant No. \mbox{RGPIN-2024-04021}}, and Matthew Harrison-Trainor\thanks{The third author was supported by a Sloan Research Fellowship and by the National Science Foundation under Grant No.\ \mbox{DMS-2419591}.}}

\newcommand{\stage}{t}
\newcommand{\att}{k}
\newcommand{\extent}{n}

\newcommand{\mc}[1]{\mathcal{#1}}

\DeclareMathOperator{\Mod}{Mod}

\newtheorem{theorem}{Theorem}[section]
\newtheorem{lemma}[theorem]{Lemma}

\newtheorem{corollary}[theorem]{Corollary}
\newtheorem{proposition}[theorem]{Proposition}

\theoremstyle{definition}

\newtheorem{question}[theorem]{Question}

\theoremstyle{remark}

\begin{document}

\maketitle

\begin{abstract}
    Given a countable mathematical structure, its \emph{Scott sentence} is a sentence of the infinitary logic $\mc{L}_{\omega_1 \omega}$ that characterizes it among all countable structures. We can measure the complexity of a structure by the least complexity of a Scott sentence for that structure. It is known that there can be a difference between the least complexity of a Scott sentence and the least complexity of a computable Scott sentence; for example, Alvir, Knight, and McCoy showed that there is a computable structure with a $\Pi_2$ Scott sentence but no computable $\Pi_2$ Scott sentence. It is well known that a structure with a $\Pi_2$ Scott sentence must have a computable $\Pi_4$ Scott sentence. We show that this is best possible: there is a computable structure with a $\Pi_2$ Scott sentence but no computable $\Sigma_4$ Scott sentence. We also show that there is no reasonable characterization of the computable structures with a computable $\Pi_n$ Scott sentence by showing that the index set of such structures is $\Pi^1_1$-$m$-complete.
\end{abstract}

\section{Introduction}

Let $\mc{A}$ be a countable mathematical structure, such a countable graph, group, or ring. Suppose that we want to characterize $\mc{A}$ by writing down a sentence, or theory, which characterizes $\mc{A}$ up to isomorphism. If we work in elementary first-order logic then, as a consequence of compactness, we cannot do this for most countable structures $\mc{A}$. However, suppose we strengthen our logic to the logic $\mc{L}_{\omega_1 \omega}$ which allows countably infinite conjunctions and disjunctions. Scott \cite{Sco65} showed that for any countable structure $\mc{A}$ there is a sentence $\varphi$ of $\mc{L}_{\omega_1 \omega}$ that characterizes $\mc{A}$ up to isomorphism among countable structures, i.e., for all countable $\mc{B}$,
\[ \mc{B} \models \varphi \Longleftrightarrow \mc{A} \cong \mc{B}.\]
We call such a sentence a \emph{Scott sentence} for $\mc{A}$. This fact implies, for example, that while isomorphism is analytic-complete if we fix a particular $\mc{L}$-structure $\mc{A}$ then the set ${\{ \mc{B} \in \Mod(\mc{L}): \mc{A} \cong \mc{B}\}}$ is actually Borel.

The standard proof that every countable structure has a Scott sentence uses the back-and-forth relations. First, one shows that they must stabilize at some countable ordinal, and then from this one can extract a Scott sentence, called the \textit{canonical Scott sentence}. This Scott analysis of a structure has played an
important role in the study of Vaught’s conjecture, e.g. in Morley’s theorem \cite{Morley70} that
the number of non-isomorphic countable models of a theory is either at most $\aleph_1$ or is exactly $2^\aleph_0$. 

For each particular structure $\mc{A}$, there is some ordinal $\alpha$ at which the back-and-forth relations stabilize. This gives a way of assigning an ordinal rank to each countable structure giving a measure of that structure's complexity. There are numerous non-equivalent definitions of Scott rank depending on the particular back-and-forth relations one chooses or exactly what one means by stabilizing. Within the last ten years the most commonly used has been the \textit{(unparametrized) Scott rank} due to Montalb\'an \cite{MonSR} which measures the least complexity of a Scott sentence for $\mc{A}$, namely, the Scott rank is the least $\alpha$ such that $\mc{A}$ has a $\Pi_{\alpha+1}$ Scott sentence. The following theorem shows that this definition is particularly robust.

\begin{theorem}[Montalb\'an \cite{MonSR}]\label{thm:robust}
    Let $\mc{A}$ be a countable structure, and $\alpha$ a countable ordinal. The following are equivalent:
    \begin{enumerate}
		\item $\mc{A}$ has a $\Pi_{\alpha+1}$ Scott sentence.
		\item Every finite tuple has an automorphism orbit in $\mc{A}$ which is $\Sigma_\alpha$-definable, i.e., $\mc{A}$ has a \textbf{Scott family} of $\Sigma_{\alpha}$ formulas (without parameters.)
        \item The set $\{\mc{B} \in \Mod(\mc{L}) : \mc{B} \cong \mc{A}\}$ of isomorphic copies of $\mc{A}$ is $\mathbf{\Pi}^0_{\alpha+1}$.
		\item $\mc{A}$ is (boldface) $\mathbf{\Delta}^0_\alpha$-categorical.
	\end{enumerate}
\end{theorem}

In computable structure theory, there has been a long history of studying the Scott analysis of computable structures, or more generally how the Scott analysis of a structure relates computationally to the structure. Nadel \cite{Nadel} showed that every computable structure $\mc{A}$ has Scott rank at most $\omega_1^{\mc{A}}+1$. Harrison \cite{Harrison68}, Millar and Knight \cite{KnightMillar} (building on work of Makkai \cite{Makkai}), Harrison-Trainor, Igusa, and Knight \cite{HTIgusaKnight18}, Harrison-Trainor \cite{HTSpectra}, and Alvir, Greenberg, Harrison-Trainor, and Turetsky \cite{AGNHTT} constructed various examples of computable structures with non-computable Scott ranks.

In particular, if $\mc{A}$ is computable and has computable Scott rank $< \omega_1^{CK}$, then the canonical Scott sentence constructed by Scott will also be a computable sentence. However this construction is not optimal: Alvir, Knight, and McCoy \cite{AlvirKnightMcCoy} showed that there is a computable structure with a $\Pi_2$ Scott sentence but with no computable $\Pi_2$ Scott sentence. It is well-known (see, e.g., Lemma VI.14 of \cite{MBook}) that if a computable structure $\mc{A}$ has a $\Pi_\alpha$ Scott sentence, with $\alpha$ computable, then it has a computable $\Pi_{2 \alpha}$ Scott sentence. This is obtained by noting that if $\mc{A}$ has a $\Pi_\alpha$ Scott sentence $\varphi$, then $\mc{A}$ has the property that for countable structure $\mc{B}$,
\[ \mc{A} \leq_\alpha \mc{B} \Longleftrightarrow \mc{A} \cong \mc{B}.\]
Writing out the definition of the back-and-forth relations, we can see that there is a computable $\Pi_{2\alpha}$ sentence $\psi$ such that
\[ \mc{B} \models \psi \Longleftrightarrow \mc{A} \leq_\alpha \mc{B}.\]
This $\psi$ is a computable $\Pi_{2\alpha}$ Scott sentence.

In this paper, we restrict to the case of structures with $\Pi_2$ Scott sentences, which are called by Montalb\'an \cite{MonEffAtomicStructures} the $\exists$-atomic structures because every automorphism orbit is isolated by a (finitary) existential formula. Even in this seemingly simple class of structures we find that there is significant complexity. As described above, such structures have a computable $\Pi_4$ Scott sentence; we improve the result of \cite{AlvirKnightMcCoy} to show that this upper bound is optimal.

\begin{restatable}{theorem}{pithree}
\label{thm:pi3}
    There is a computable structure with a $\Pi_2$ Scott sentence but no computable $\Sigma_4$ Scott sentence.
\end{restatable}

\noindent By taking Marker extensions / jump inversions, one can also obtain that for any $n$ there is a computable structure with a $\Pi_{n}$ Scott sentence but no computable $\Sigma_{n + 2}$ Scott sentence. (One can also extend this to non-limit ordinals though we will remain at finite levels since all of the complexity appears there; see Corollary \ref{cor:nopi3}.) We leave open the question of improving this to an example with no $\Sigma_{2 n}$ Scott sentence.

In the second part of this paper, we consider the \emph{effective Scott rank}, i.e., the least $\alpha$ such that a given structure $\mc{A}$ has a computable $\Pi_{\alpha + 1}$ Scott sentence (when a computable Scott sentence exists). Whether or not this effective Scott rank is as robust as the (non-effective) Scott rank was an open problem from \cite{AlvirKnightMcCoy}. It is known that there are a few conditions equivalent to having a formally $\Sigma_{\alpha}$ Scott family; a ``boldface" version of this result led to many of the equivalences of \Cref{thm:robust}, so it was not unreasonable to expect similar robustness in the effective setting. In fact, it was known that for a computable structure, having having a formally $\Sigma_{\alpha}$ Scott family implies that the structure has a computable $\Pi_{\alpha+1}$ Scott sentence. Whether the reverse implication held was left open. 

In Section 4, we show that the effective Scott rank is not robust and that there is no good way to characterize when a structure has a computable $\Pi_n$ Scott sentence, even when the structure is computable.

\begin{restatable}{theorem}{complete}\label{thm:complete}
    With $(\mc{A}_i)_{i \in \omega}$ an effective list of (possibly partial) structures in a sufficiently rich language, the set
    \[ \{i \mid \text{$\mc{A}_i$ has a computable $\Pi_2$ Scott sentence}\}\]
    is $\Pi^1_1$-$m$-complete.
\end{restatable}

\noindent By a sufficiently rich language we mean any language which is universal, e.g., those including at least a binary relation symbol. This result was also obtained independently and around the same time by Knight, Lange, and McCoy \cite{KLM}.

Again by taking Marker extensions / jump inversions, one can replace $\Pi_2$ with $\Pi_n$, or through the hyperarithmetic hierarchy, as in Corollary \ref{cor:complete}. One can obtain from this theorem several interesting corollaries which we give in Section \ref{sec:cor}, for example:

\begin{corollary}
    There is a computable structure $\mc{A}$ with a $\Pi_2$ Scott sentence but no computable $\Pi_2$ Scott sentence, but with a computable $\Pi_2$ sentence $\varphi$ such that, for all hyperarithmetic $\mc{B}$,
    \[ \mc{B} \models \varphi \;\Longleftrightarrow\; \mc{A} \cong \mc{B}.\]
\end{corollary}

We leave open several questions. Though there are further generalizations that one can ask (e.g., to infinite ordinals), we see the two mains questions to be the following.

\begin{question}\label{q:generalize}
    Is there, for every $n$, a computable structure $\mc{A}$ with a $\Pi_n$ Scott sentence but with no computable $\Sigma_{2n}$ Scott sentence?
\end{question}

\begin{question}
    Is the set of computable structures with a $\Pi_2$ Scott sentence and a computable $\Pi_3$ Scott sentence $\Pi^1_1$-$m$-complete?
\end{question}

For Question \ref{q:generalize}, recall that the upper bound of $\Pi_{2n}$ was obtained as follows. Given $\mc{A}$ with a $\Pi_n$ Scott sentence, $\mc{A} \leq_n \mc{B} \Longrightarrow \mc{A} \cong \mc{B}$. We find a $\Pi_{2n}$ sentence $\psi$ such that $\mc{B} \models \psi \Longleftrightarrow \mc{A} \leq_n \mc{B}$. Chen, Gonzalez, and Harrison-Trainor \cite{ChenGonzalezHT} have recently showed that the set of pairs $\{(\mc{A},\mc{B}) : \mc{A} \leq_n \mc{B}\}$ is $\mathbf{\Pi}^0_{2n}$-complete. This means that there is no way of defining the set $\{ \mc{B} : \mc{A} \leq_n \mc{B}\}$ in a way that is both better than $\Pi_{2n}$ and which is ``schematic'' in $\mc{A}$. On the other hand, they showed that the set $\{ \mc{B} : \mc{A} \leq_n \mc{B}\}$ is $\mathbf{\Pi}^0_{n+2}$. This  relies on the fact that every $\Pi_n$-type in $\mc{A}$ is $\Pi_n$-definable, and so is a non-effective argument. The case of Question \ref{q:generalize} solved in this paper, $n = 2$, satisfies $n+2=2n = 4$ and the distinction does not yet show up at this level. Thus one might expect Question \ref{q:generalize} for $n\geq 3$ to involve some new insight.

\section{A simplifying remark}

Given a structure $\mc{A}$, we can consider the structure $\mc{A} \cdot \omega$ which consists of an equivalence relation $E$ with infinitely many equivalence classes, on each of which is a copy of $\mc{A}$. Then:
    \begin{enumerate}
        \item if $\mc{A}$ has a (computable) $\Pi_\alpha$ Scott sentence, then $\mc{A} \cdot \omega$ will have a (computable) $\Pi_\alpha$ Scott sentence, and
        \item if $\mc{A}$ has no (computable) $\Pi_\alpha$ Scott sentence, then $\mc{A} \cdot \omega$ will have no (computable) $\Sigma_{\alpha + 1}$ Scott sentence.
    \end{enumerate}
    (1) is straightforward, and (2) uses the fact that $\mc{B}$ has a $\Sigma_{\alpha+1}$ Scott Sentence if and only if there is $\bar{b} \in \mc{B}$ such that $(\mc{B},\bar{b})$ has a $\Pi_{\alpha}$ Scott sentence. (This fact was first stated by Montalb\'an in \cite{MonSR} and proved in \cite{MonEffAtomicStructures,AGNHTT}.)

    Given the above facts, we do not need to consider $\Sigma$ Scott sentences, as e.g., the fact that there is a computable structure with a $\Pi_2$ Scott sentence but no computable $\Pi_2$ Scott sentence yields that there is a structure with a $\Pi_2$ Scott sentence but no computable $\Sigma_3$ Scott sentence. When we prove Theorem \ref{thm:pi3} we will prove that there is a computable structure with a $\Pi_2$ Scott sentence but no computable $\Pi_3$ Scott sentence; it will follow that there is a computable structure with a $\Pi_2$ Scott sentence but no computable $\Sigma_4$ Scott sentence.


\section{A computable structure with a $\Pi_2$ Scott sentence but no computable $\Pi_3$ Scott sentence}

As a warmup we sketch a new construction of a computable structure with a $\Pi_2$ Scott sentence but no computable $\Pi_2$ Scott sentence. Theorems \ref{thm:pi3} and \ref{thm:complete} will build on this technique and so we begin by presenting it in its simplest form.

\begin{theorem}[Alvir, Knight, and McCoy \cite{AlvirKnightMcCoy}]
    There is a computable structure with a $\Pi_2$ Scott sentence but with no computable $\Pi_2$ Scott sentence.
\end{theorem}
\begin{proof}[Proof sketch]
        We list all of the computable $\Pi_2$ sentences as $(\theta_e)_{e \in \omega}$ where
    \[ \theta_e = \bigdoublewedge_{i \in \omega} \forall \bar{x}_{e,i} \;\; \varphi_{e,i}(\bar{x}_{e,i}) \]
    where the $\varphi_{e,i}(\bar{x}_{e,i})$ are computable $\Sigma_1$ formulas uniformly in $e,i$ and the arities of each of the $\bar{x}_{e,i}$ are also computable in the indices.

   We can take $\mc{A}$ to be the ``bouquet graph'' $\mc{G}_1(\mc{F})$ of a collection $\mc{F}$ of subsets of $\omega$. $\mc{A}$ will be $\exists$-atomic and hence have a $\Pi_2$ Scott sentence. Recall that by Lemma 8.17 of  \cite{Montalbán_2021} the structure $\mc{G}_1(\mc{F})$ is $\exists$-atomic exactly when $\mc{F}$ is discrete.

    The structure $\mathcal{A}$ will consist of a number of elements each of which is given various c.e.\ labels. We can think of $\mc{A}$ as consisting of the elements which are the centers of the flowers, labeled by labels $\ell_n$; putting a label $\ell_n$ on a flower means to add a loop of length $n+3$ to the flower.

    We divide our label into \textit{sort labels} $(u_e)_{e \in \omega}$ such that only exactly one label holds of each element, dividing the domain into the disjoint sets $U_e = \{ x \in A : u_e(x) \}$; we think of these as different sorts of the structure, and call the elements of $U_e$ the \textit{$e$th sort}. Then we have two other sets of labels $(\ell_i)_{i \in \omega}$ and $(\ell^\dagger_i)_{i \in \omega}$ which we call simply \textit{labels}. For example, we may use loops of length $2e + 3$ for the labels $u_e$, and of length $4i + 4$ for the $\ell_i$, and $4i + 6$ for the labels $\ell^\dagger_i$.

    We use the $e$th sort to diagonalize against $\theta_e$. We do this by simultaneously building a structure $\mc{B}_e$ such that if $\mc{A} \models \theta_e$ then $\mc{B}_e \models \theta_e$ and $\mc{A} \ncong \mc{B}$. $\mc{B}_e$ will differ from $\mc{A}$ only on the $e$th sort $U_e$. We build $\mc{A}$ by finite approximations $\mc{A} = \bigcup_s \mc{A}_s$. $\mc{B}_e$ will also be built by approximations $\mc{B}_e =\bigcup_s \mc{B}_{e,s}$ with each $\mc{B}_{e,s} \cong \mc{A}_s$.

    When we describe the construction of $\mc{A} = \bigcup \mc{A}_s$, we will describe the construction of the $e$th sort. The constructions for the different sorts should be thought of as happening simultaneously. 

    During the construction certain stages will be \textit{$e$-expansionary stages} where we get evidence that $\mc{A}_s \models \theta_e$. We use $k = k[s]$ to keep track of the number of expansionary stages and $n = n[s] = k[s] + 1$ to keep track of the number of flowers in $\mc{A}_s$. At stage $s$, $\mc{A}_s$ will have elements $a_1,\ldots,a_n$ and $\mc{B}_{e,s}$ will have elements $b_1,\ldots,b_{n-1},c$.
    
\medskip

\noindent \textit{Construction.}

\medskip

\noindent \textit{Stage 0.} We begin with $n = \extent[0] = 1$ and $k = k[0] = 0$. $\mc{A}_0$ consists of a single element $a_1$, and $\mc{B}_{e,0}$ of a single element $c$. We give $a_1$ and $c$ the same label $\ell_0$.

\medskip

\noindent \textit{Stage $s+1$.} Suppose that we have constructed $\mc{A}_s$ and $\mc{B}_{e,s}$ with $k = k[s]$ and $n = n[s] = k+1$. The elements of $\mc{A}_s$ will be $a_1,\ldots,a_n$. For each $i \leq n$, $a_i$ will have the labels $\ell_j$ for $j \leq i$; and for $i < n$, $a_i$ will have the label $\ell^\dagger_i$. The elements of $\mc{B}_{e,s}$ will be $b_1,\ldots,b_{n-1},c$. For each $i \leq n-1$, $b_i$ will have the labels $\ell_j$ for $j \leq i$ and $\ell_i^\dagger$. The element $c$ will have the labels $\ell_j$ for $j \leq n$.

Let $s_k$ be the last expansionary stage. Check whether
\[ \mc{A}_s \models \bigwedge_{i \leq k } \forall \bar{x}_{e,i} \in \mc{A}_{s_k} \;\; \varphi_{e,i}(\bar{x}_{e,i}).\]
If not, then this is not an expansionary stage. Make no changes to $\mc{A}$, $\mc{B}_e$, $k[s+1] = k[s]$, or $n[s+1] = n[s]$. If so then this stage is an expansionary stage. In $\mc{A}_{s+1}$ put the label $\ell_i^\dagger$ on $a_n$, and add a new element $a_{n+1}$ to with labels $\ell_i$ for $i \leq n+1$. In $\mc{B}_{e,s+1}$ add a new element $b_n$ with labels $\ell_i$ for $i \leq n$ and $\ell_n^\dagger$. Put the label $\ell_{n+1}$ on $c$. Set $k[s+1] = k[s] + 1$ and $n[s+1] = n[s] + 1$.

\medskip

\noindent \textit{End construction.}

\medskip

We will not give the full details of the verification. However, the general idea of the argument is as follows. We argue that there are infinitely many expansionary stages if and only if $\mc{A} \models \theta_e$. On the one hand, if there are infinitely many expansionary stages $s$ then for every $k$ there is $s$ such that
\[ \mc{A}_s \models \bigwedge_{i \leq k } \forall \bar{x}_{e,i} \in \mc{A}_{s_k} \;\; \varphi_{e,i}(\bar{x}_{e,i}).\]
and so as $\varphi_{e,i}$ is $\Sigma_1$ we have $\mc{A} \models \theta_e$. On the other hand, if there are only finitely many expansionary stages, then there is $i$ and $\bar{x}$ such that for all $s$ we have
\[ \mc{A}_s \models \neg \varphi_{e,i}(\bar{x}).\]
Then $\mc{A} \models \neg \theta_e$.

If there are infinitely many expansionary stages, then $\mc{B}_{e} \ncong \mc{A}$ as the element $c \in \mc{B}_e$ has infinitely many labels $\ell_i$ for each $i$, but $\mc{A}$ has no such element. We can also argue that $\mc{B}_{e,s} \models \theta_n$; the tricky part is to check tuples $\bar{x}_{e,i}$ containing $c$. Here we use the fact that at the $k+1$st expansionary stage $s > s_k$ we have
\[ \mc{A}_s \models \bigwedge_{i \leq k } \forall \bar{x}_{e,i} \in \mc{A}_{s_k} \;\; \varphi_{e,i}(\bar{x}_{e,i}) \]
and that that the isomorphism $\mc{A}_s \cong \mc{B}_{e,s}$ extends the isomorphism $\mc{A}_{s_k} \cong \mc{B}_{e,s_k}$.

Finally, we need to check that $\mc{A}$ is always $\exists$-atomic and hence always has a $\Pi_2$ Scott sentence. Any element with a label $\ell_i^\dagger$ is isolated by that label in its sort; and in each sort, there is at most one element without such a label (and only if there were finitely many expansionary stages in that sort), and that element is isolated by a label $\ell_i$ which no other element of the sort has.
\end{proof}

The construction of a computable structure with a $\Pi_2$ Scott sentence but no computable $\Pi_3$ Scott sentence is similar, but the guessing as to whether a $\Pi_3$ sentence is true in $\mc{A}$ is more involved. Otherwise, the general ideas in the argument remain the same. Recall that by proving that such a structure exists, we have also shown that there is a computable structure with a $\Pi_2$ Scott sentence but no computable $\Sigma_4$ Scott sentence.

\pithree*


\begin{proof}
    We list all of the computable $\Pi_3$ sentences as $(\theta_e)_{e \in \omega}$ where
    \[ \theta_e = \bigdoublewedge_{i \in \omega} \forall \bar{x}_{e,i} \; \bigdoublevee_{j \in \omega} \exists \bar{y}_{e,i,j} \;\; \varphi_{e,i,j}(\bar{x}_{e,i},\bar{y}_{e,i,j}) \]
    where the $\varphi_{e,i,j}(\bar{x}_{e,i},\bar{y}_{e,i,j})$ are computable $\Pi_1$ formulas uniformly in $e,i,j$ and the arities of each of the $\bar{x}_{e,i}$ and $\bar{y}_{e,i,j}$ are also computable in the indices.

    We will build a structure $\mc{A}$ which is $\exists$-atomic. The structure $\mc{A}$ consist of a number of elements each of which is given various c.e.\ labels. Formally, $\mc{A}$ can be taken to be the ``bouquet graph'' $\mc{G}_1(\mc{F})$ of a collection $\mc{F}$ of subsets of $\omega$---see Section VIII of \cite{MBook}. As we construct $\mc{A}$, we think of ourselves as giving the centers of various ``flower graphs'' and when we add labels this corresponds to adding loops to the flower graphs. Thus in a computable construction of $\mc{A}$ we can add labels to elements in a c.e.\ way.
    
    We introduce two sets of distinguished labels. First, we have \textit{sort labels} $(u_e)_{e \in \omega}$ such that only exactly one label holds of each element, dividing the domain into the disjoint sets $U_e = \{ x \in A : u_e(x) \}$; we think of these as different sorts of the structure, and call the elements of $U_e$ the \textit{$e$th sort}. The second set of labels $(\ell_i)_{i \in \omega}$ we will just call \textit{labels}, and they will be the most important labels for the construction.

Within $U_e$ we will diagonalize against $\theta_e$. We will do this by constructing another countable structure $\mc{B}_e$ such that if $\mc{A} \models \theta_e$ then $\mc{B}_e \models \theta_e$ and $\mc{A} \ncong \mc{B}_e$. $\mc{B}_e$ will be the same as $\mc{A}$ except for the $e$th sort $U_e$ on which they will differ. At each stage $s$ we will have a finite approximation $\mc{A}_s$ to $\mc{A}$, with $\mc{A} = \bigcup_s \mc{A}_s$. $\mc{B}_e$ will also be built by finite approximations $\mc{B}_e =\bigcup_s \mc{B}_{e,s}$ with each $\mc{B}_{e,s} \cong \mc{A}_s$. Thus one can think of $\mc{A}$ and $\mc{B}_e$ as direct limits of direct systems of the same finite structures but with different embeddings
\[ \mc{A}_0 \to \mc{A}_1 \to \mc{A}_2 \to \cdots .\]
At each stage $s$, each element of $\mc{A}_s$ will satisfy some finite set of labels (in fact two, a domain label and another label) that is satisfied by no other elements except duplicates of itself which satisfy exactly the same labels. Thus the isomorphism $\mc{A}_s \to \mc{B}_{e,s}$ will be easily and uniquely determined up to mapping duplicates.

For each $e$, we have requirements $\mc{R}^{e}_{i,\bar{b}}$ for each $i \in \omega$ and $\bar{b} \in \mc{B}_e$ of the right arity; to make sense of this, we must fix ahead of time the domain of $\mc{B}_e$ even though these elements will only be added to $\mc{B}_e$ slowly over time. The elements $\bar{b}$ may be in any sort of $\mc{B}_e$, not just the $e$th. Each requirement will have a stage $\stage = \stage(\mc{R}^{e}_{i,\bar{b}})$ when it began working; it will only start working at some stage when $\bar{b} \in \mc{B}_{e,\stage}$. At this stage, there will be an element $\bar{a}$ of $\mc{A}_\stage$ which corresponds to $\bar{b} \in \mc{B}_{e,\stage}$ at this stage $\stage$ via the isomorphism $\mc{A}_\stage \cong \mc{B}_{e,\stage}$. The requirement will be met if whenever $\mc{A} \models \bigdoublevee_j \exists \bar{y}_{e,i,j} \; \varphi_{e,i,j}(\bar{a},\bar{y}_{e,i,j})$ then $\mc{B}_e \models \bigdoublevee_j \exists \bar{y}_{e,i,j} \; \varphi_{e,i,j}(\bar{b},\bar{y}_{e,i,j})$; equivalently, if $\mc{B}_e \models \bigdoublewedge_j \forall \bar{y}_{e,i,j} \; \neg \varphi_{e,i,j}(\bar{b},\bar{y}_{e,i,j})$ then $\mc{A} \models \bigdoublewedge_j \forall \bar{y}_{e,i,j} \; \neg \varphi_{e,i,j}(\bar{a},\bar{y}_{e,i,j})$. One can think of the requirements for a given $e$ as splitting up, into the different possible outermost witnesses, the supreme requirement that if $\mc{A} \models \theta_e$ then $\mc{B}_e \models \theta_e$ (or equivalently if $\mc{B} \models \neg \theta_e$ then $\mc{A} \models \neg \theta_e$).  The requirement will require attention when we believe that $\mc{B}_e \models \forall \bar{y}_{e,i,j} \; \neg \varphi_{e,i,j}(\bar{b},\bar{y}_{e,i,j})$. This is $\Pi^0_2$ and so if this is the case, this requirement will require attention infinitely often. Meeting a single requirement in this infinitary manner will be enough, as for $\bar{b} \in \mc{B}_e$ we have that if $\mc{B}_e \models \forall \bar{y}_{e,i,j} \; \neg \varphi_{e,i,j}(\bar{b},\bar{y}_{e,i,j})$ then $\mc{B}_e \nmodels \theta_e$ and the requirement will have the outcome that $\mc{B}_e \cong \mc{A}$ so that $\mc{A} \nmodels \theta_e$. Thus even though the injury may be infinite, any guessing is simply finitary since if a requirement is injured infinitely often, then it does not need to be satisfied as a higher priority requirement will be satisfied. In addition to the value $\stage = \stage(\mc{R}^{e}_{i,\bar{b}})$, each requirement will also have a value $\att = \att(\mc{R}^{e}_{i,\bar{b}})$ which denotes the number of times that it has previously received attention.

The general idea of the construction is that, over time, we work to build $\mc{B}_e$ taking, at each stage, a step towards making $\mc{B}_e \ncong \mc{A}$. In particular, there will be some special element of the $e$th sort of $\mc{B}_e$ which we will work to make different from each element of $\mc{A}$. However, whenever we see evidence that $\mc{B}_e \models \neg \theta_e$, we will undo our previous work to make $\mc{B}_e$ different from $\mc{A}$; since this makes $\mc{B}_e$ look similar to $\mc{A}$, we can at the same time transfer our evidence that $\mc{B}_e \models \neg \theta_e$ to evidence that $\mc{A} \models \neg \theta_e$. If in fact $\mc{B}_e \models \neg \theta_e$, then we will continually find evidence of this, continually rolling back our construction, so that actually $\mc{B}_e$ never becomes different from $\mc{A}$. Since $\mc{A} \cong \mc{B}_e$, $\mc{A} \models \neg \theta_e$. On the other hand, if $\mc{B}_e \models \theta_e$, then we will have $\mc{A} \ncong \mc{B}_e$ so that $\theta_e$ cannot be a Scott sentence for $\mc{A}$.

We describe the construction of $\mc{A}$ and the $\mc{B}_e$ stage-by-stage. We will describe for a fixed $e$ the construction of the $e$th sort of $\mc{A}$ together with the $e$th sort of $\mc{B}_e$. However all of these constructions for the different sorts should be thought of as happening simultaneously. All of the other sorts of $\mc{B}_e$ will be exactly the same as $\mc{A}$, so when constructing the $e$th sort of $\mc{A}$ we need only discuss $\mc{B}_e$ and not any other $\mc{B}_m$. In the construction that follows, without explicitly saying so we put the label $u_e$ on all elements that we add to $\mc{A}$ and $\mc{B}_e$ so that they are in $U_e$. We sometimes suppress $e$ and the label $u_e$ in what follows, e.g., we write for $\theta = \theta_e$
\[ \theta = \bigdoublewedge_{i \in \omega} \forall \bar{x}_{i} \; \bigdoublevee_{j \in \omega} \exists \bar{y}_{i,j} \;\; \varphi_{i,j}(\bar{x}_{i},\bar{y}_{i,j}).\]
However, when we are guessing at whether some formula is true in $\mc{A}$ or $\mc{B}_e$, we must allow the quantifiers to quantify over all of the sorts; we will remind the reader of this later.

We begin at stage $0$ with a single element $a_0$ with no labels (other than $u_e$, which we no longer mention). At each stage $s$ we will have a number $\extent = \extent[s]$. There will be $\extent$ \textit{active elements} designated $a_1[s],\ldots,a_\extent[s]$. Each active element will have a label that no other active element has. The rest of the elements of $\mc{A}_s$ will be \textit{duplicate elements} and will have exactly the same labels as one of the active elements. When we declare a (formerly active) element to be a duplicate of $a_i$, we will give it exactly the same labels as $a_i$, and whenever $a_i$ gets a new label, the element will as well; thus if later $a_i$ becomes a duplicate of some other $a_j$, then any duplicates of $a_i$ will become duplicates of $a_j$. When an active element $a_i$ is declared to be a duplicate of another, becoming inactive, it gives up the designation $a_i$. Recall that $\mc{B}_{e,s}$ will be isomorphic to $\mc{A}_s$; however, $\mc{B}_e$ will follow different embeddings from one stage to the next. The active elements of $\mc{B}_{e,s}$ will be $b_1[s],\ldots,b_{\extent-1}[s],c$. The isomorphism $\mc{A}_s \to \mc{B}_{e,s}$ will map $a_i \mapsto b_i$ and $a_\extent \mapsto c$. (The duplicate elements will also be mapped correspondingly.)

To illustrate the idea of the construction, we will describe a generic stage of the construction when no requirement requires attention. Suppose that we have defined $\mc{A}_s$ and $\mc{B}_{e,s}$ with $\extent = \extent[s]$. Define $\mc{A}_{s+1}$ as follows. Add a new element $a_{\extent+1}$, and give $a_{\extent+1}$ all of the labels that $a_\extent$ had in $\mc{A}_s$. Give $a_\extent$ and $a_{\extent+1}$ each a new label unique only to it. For $\mc{B}_{e,s+1}$, add a new element $b_\extent$ and let $b_\extent$ have the same labels as $a_\extent$, and let $c$ have the same labels as $a_{\extent+1}$. If we do this at every stage, then $\mc{A}$ will have elements $a_1,a_2,\ldots$ and $\mc{B}_e$ will have corresponding elements $b_1,b_2,\ldots$ with the same labels, but also an additional element $c$ which does not correspond to any element of $\mc{A}$. Thus $\mc{A} \ncong \mc{B}_e$. However we have not ensured that if $\mc{A} \models \theta_e$ then $\mc{B}_e \models \theta_e$.

To ensure that if $\mc{A} \models \theta_e$ then $\mc{B}_e \models \theta_e$, our strategy (for some $i$ and $\bar{b}$) is to wait for stages at which we think that  $\mc{B}_{e,s} \models \bigdoublewedge_j \forall \bar{y}_{i,j} \; \neg \varphi^e_{i,j}(\bar{b},\bar{y}_{i,j})$. This is evidence that $\mc{B}_e \nmodels \theta_e$. At stage $s+1$ instead of taking our usual action towards making $\mc{B}_e$ different from $\mc{A}$, we instead revert some of our previous actions to make $\mc{B}_e$ look like $\mc{A}$. If $\mc{B}_e \nmodels \theta_e$ then we will get more and more evidence of this, and we will instead have $\mc{B}_e \cong \mc{A}$ and thus $\mc{A} \nmodels \theta_e$. How we do this is the heart of the construction.

\medskip

\noindent \textit{Construction.}

\medskip

\noindent \textit{Stage 0.} We begin with $\extent = \extent[0] = 1$ and $\mc{A}_0$ consisting of a single element $a_1$, and $\mc{B}_{e,0}$ of a single element $c$. We give $a_1$ a label and give $c$ the same label.

\medskip

\noindent \textit{Stage $s+1$.} Suppose that we have constructed $\mc{A}_s$ and $\mc{B}_{e,s}$ with $\extent = \extent[s]$ and active elements $a_1,\ldots,a_\extent$. We first check whether any requirement $\mc{R}^{e}_{i,\bar{b}}$ requires attention.

Given $\mc{R}^{e}_{i,\bar{b}}$, let $t = t(\mc{R}^{e}_{i,\bar{b}})$ be the corresponding stage and let $k = k(\mc{R}^{e}_{i,\bar{b}})$. Then we say that $\mc{R}^{e}_{i,\bar{b}}$ \textit{requires attention} if \[\mc{B}_{e,s} \models \bigwedge_{j \leq \att} \forall \bar{y}_{i,j} \in \mc{B}_{e,\stage+\att} \;\; \neg \varphi^e_{i,j}(\bar{b},\bar{y}_{i,j}).\]
Note that this quantifier is not restricted to the $e$th sort of $\mc{B}_{e,\stage + \att}$ but is over the elements of all sorts. 

If no requirement requires attention, then define $\mc{A}_{s+1}$ as follows. Add a new element $a_{\extent+1}$, and give $a_{\extent+1}$ all of the labels that $a_\extent$ had in $\mc{A}_s$. Give each of $a_\extent$ and $a_{\extent+1}$ a new label unique only to it. Any duplicates of $a_\extent$ receive the same labels that $a_\extent$ did. For $\mc{B}_{e,s+1}$, add a new element $b_\extent$ and let $b_\extent$ have the same labels as $a_\extent$, and let $c$ have the same labels as $a_{\extent+1}$. There may be elements of $\mc{B}_{e,s}$ which look like duplicates of $c$; these are the elements which correspond to the duplicates of $a_\extent$. We call these \textit{pending duplicates} and we now give them the same labels as $b_\extent$, making them duplicates of $b_\extent$. Let $\mc{R}^{e}_{i,\bar{b}}$ be the next highest priority requirement; we may assume, by appropriately listing the requirements, that $\bar{b} \in \mc{B}_{e,s+1}$. We initialize $\mc{R}^{e}_{i,\bar{b}}$ by setting $\stage(\mc{R}^{e}_{i,\bar{b}}) = s+1$ and $\att(\mc{R}^{e}_{i,\bar{b}}) = 0$.

Otherwise, let $\mc{R}^{e}_{i,\bar{b}}$ be the highest priority requirement which requires attention, and we will act on the requirement as follows. Let $\stage = \stage(\mc{R}^{e}_{i,\bar{b}})$ be the corresponding stage, and $\extent^* = \extent[\stage]$ the corresponding value of $\extent$. Let $\extent[s+1] = \extent[\stage] = \extent^*$. In $\mc{A}_{s+1}$, declare $a_{\extent^*+1}[s],\ldots,a_{\extent}[s]$ to be duplicates of $a_{\extent^*}$; they give up these designations $a_{\extent^*+1},\ldots,a_{\extent}$. Moreover, any duplicates of these elements also become duplicates of $a_{\extent^*}$. Add labels to these elements and $a_{\extent^*}$ so that they have the same labels, i.e., for each label of one of these elements, give it to each other element. So in $\mc{A}_{s+1}$ we have active elements $a_1[s+1],\ldots,a_{\extent^*}[s+1]$ together with duplicates of these elements. In $\mc{B}_{e,s+1}$ add each label of any of $b_{\extent^*},\ldots,b_{\extent-1},c$ to each of the others. In $\mc{B}_{e,s+1}$ there are elements which correspond to the duplicates of $a_\extent[s+1]$. One should think of these not as duplicates of $c$, but as duplicates of an element $b_\extent$ which does not yet exist. We call these elements \textit{pending duplicates}. Then put $b_i[s+1] := b_i[s]$ for $i < \extent^*$. Thus, at stage $s+1$, the correspondence between $\mc{A}_{s+1}$ and $\mc{B}_{e,s+1}$ is $a_i \mapsto b_i$ for $i < \extent^*$ and $a_{\extent^*} \mapsto c$ (together with the duplicates mapping to corresponding duplicates, with duplicates of $a_\extent$ mapping to pending duplicates). Each lower priority $e$-requirement is injured, and increment $\att(\mc{R}^{e}_{i,\bar{a}})$ by $1$.

\medskip

\noindent \textit{End construction.}

\medskip

We must now verify that the construction works. Standard arguments show that for each $e$, there are two possibilities. First, it might be that some requirement $\mc{R}^e_{i,\bar{b}}$ is, after some point, never injured and acts infinitely often. In this case, let $
\stage = \stage(\mc{R}^e_{i,\bar{b}})$ and let $\extent = \extent[\stage]$. Since no higher priority requirement acts, the elements $a_1,\ldots,a_{\extent}$ never become duplicates of any other elements, and every time $\mc{R}^e_{i,\bar{b}}$ acts every other element which is not already becomes a duplicate of one of these elements. Thus the $e$th sort of $\mc{A}$ consists of $\extent$ elements $a_1,\ldots,a_{\extent}$ together with duplicates of them. We write $a_1[\infty],\ldots,a_{\extent}[\infty]$ for these elements.

The other possibility is that each requirement acts only finitely many times. If, after some stage $s$ with $\extent = \extent[s]$, no requirement with $\stage(\mc{R}^e_{i,\bar{b}}) < s$ ever acts, then after this stage $a_1,\ldots,a_\extent$ never become duplicates. In the limit, we build infinitely many elements $a_1[\infty],a_{2}[\infty],a_3[\infty],\ldots$. (Note that the first time we introduce some $a_i$, this element may become a duplicate; but there will be some point after which this stops happening.) All of the other elements will be duplicates of these.

\begin{lemma}
    $\mc{A}$ has a $\Pi_2$ Scott sentence.
\end{lemma}
\begin{proof}
Recall that by Lemma 8.17 of \cite{Montalbán_2021} the structure $\mc{G}_1(\mc{F})$ is $\exists$-atomic exactly when $\mc{F}$ is discrete. Therefore it suffices to show that for each element of $\mc{A}$, there is a finite set of labels on that element such that no other element has the same labels. Each element is in only one sort $U_e$, so it suffices to only compare elements within the same sort. Within each sort we will argue that each element has a label, that we call the distinguishing label, which it has and no other element has other than duplicates of that element.
    
    Each element $a_i[\infty]$, when it is created, is given a label which is never given to $a_1[\infty],\ldots,a_{i-1}[\infty]$. If $a_{i+1}[\infty]$ exists then, when it is created, $a_i[\infty]$ is given a label that is never given to $a_{i+1}[\infty]$; and no $a_j[\infty]$, $j > i+1$, is ever given that label. The other elements of $\mc{A}$ are duplicates of the $a_i[\infty]$ and hence share the same distinguishing label as the element they duplicate. 
\end{proof}

\begin{lemma}
    Fix $i$ and $\bar{b}$. Suppose that there is a stage $s$ after which $\mc{R}^e_{i,\bar{b}}$ is never injured. Let $\stage = \stage(\mc{R}^e_{i,\bar{b}})$. Suppose that 
    \[ \mc{B}_e \models \bigdoublewedge_j \forall \bar{y}_{i,j} \;\; \neg \varphi_{i,j}(\bar{b},\bar{y}_{i,j}).\]
    Then there are infinitely many stages at which $\mc{R}^e_{i,\bar{b}}$ requires attention.
\end{lemma}
\begin{proof}
    Given $s_0 = s$, we will argue that there are infinitely many stages $s_u$ at which $\mc{R}^e_{i,\bar{b}}$ requires attention. Given $s_u$, we have that 
    \[ \mc{B}_e \models \bigdoublewedge_j \forall \bar{y}_{i,j} \;\; \neg \varphi_{i,j}(\bar{b},\bar{y}_{i,j}).\]
    and so there must be some stage $s_{u+1} > s_u$ such that
    \[\mc{B}_{e,s_{u+1}} \models \bigwedge_{j \leq u+1} \forall \bar{y}_{i,j} \in B_{e,s_u} \;\; \neg  \varphi_{i,j}(\bar{b},\bar{y}_{i,j}).\]
    At this stage $s_{u+1}$, $\mc{R}^e_{i,\bar{b}}$ requires attention. (Since $\mc{R}^e_{i,\bar{b}}$ is not injured, no higher priority requirement can require attention at this stage.)
\end{proof}

\begin{lemma}
    Suppose that $\mc{B}_e \models \neg \theta_e$. Then $\mc{A} \cong \mc{B}_e$ and $\mc{A} \models \neg \theta_e$
\end{lemma}
\begin{proof}
    Then there is $i$ and $\bar{b} \in \mc{B}_e$ such that
    \[ \mc{B}_e \models \bigdoublewedge_j \forall \bar{y}_{i,j} \;\; \neg \varphi_{i,j}(\bar{b},\bar{y}_{i,j}).\]
    By the previous lemma, for any such $i$ and $\bar{b}$, if there is some stage after which it is not injured, the requirement $\mc{R}^e_{i,\bar{b}}$ will require attention infinitely many times. Thus in particular there is some highest priority requirement $\mc{R}^e_{i,\bar{b}}$ which requires attention infinitely often.

    Let $\stage = \stage(\mc{R}^e_{i,\bar{b}})$. Let $\extent = \extent[\stage]$. Let $s_0$ be the stage at which $\mc{R}^e_{i,\bar{b}}$ is initiated, and let $s_1,s_2,\ldots$ be the stages at which it requires attention. At each of these stages, $\mc{A}_{s_i}$ consists of $a_1,\ldots,a_e$ and duplicates, and $\mc{B}_{e,s_i}$ consists of $b_1,\ldots,b_{n-1},c$ and duplicates; thus the isomorphisms  $a_1,\ldots,a_e \mapsto b_1,\ldots,b_{n-1},c$ at these stages extend to an isomorphism $\mc{A} \cong \mc{B}_e$. (For this, we also need to know that corresponding elements have the same number of duplicates; this is maintained at each stage of the construction.) Thus $\mc{A} \cong \mc{B}_e$, and since $\mc{B}_e \models \neg \theta_e$ we have $\mc{A} \models \neg \theta_e$.
\end{proof}

As a result, if $\mc{A} \models \theta_e$ then $\mc{B}_e \models \theta_e$. Now we show that if $\mc{B}_e \models \theta_e$ then $\mc{A} \ncong \mc{B}_e$.

\begin{lemma}
    Suppose that $\mc{B}_e \models \theta_e$. Then $\mc{A} \ncong \mc{B}_e$.
\end{lemma}
\begin{proof}
        We will argue that each requirement, after it is initialized and unless it is injured, requires attention only finitely many times. We argue by induction; given a requirement $\mc{R}^e_{i,\bar{b}}$, suppose that the higher priority requirements only require attention finitely many times. Then there is a stage $s_0$ after which $\mc{R}^e_{i,\bar{b}}$ is never injured. Let $\stage = \stage(\mc{R}^e_{i,\bar{b}})$. 
          Whenever $\mc{R}^e_{i,\bar{b}}$ acts for the $\att$th time at stage $s+1 $, it is because
    \[\mc{B}_{e,s} \models \bigwedge_{j \leq \att} \forall \bar{y}_{i,j} \in \mc{B}_{e,\stage+\att} \;\; \neg \varphi_{i,j}(\bar{b},\bar{y}_{i,j}).\]
    Then if $\mc{R}^e_{i,\bar{b}}$ requires attention infinitely many times, we see that $\mc{B}_{e} \models \bigdoublewedge_j \forall \bar{y}_{i,j} \; \neg \varphi_{i,j}(\bar{b},\bar{y}_{i,j})$ and so $\mc{B}_e \models \neg \theta_e$, contradicting the hypotheses of this lemma. Thus each requirement requires attention only finitely many times.

    In particular, there are infinitely many ``true expansionary stages'' and $\mc{A} \ncong \mc{B}_e$ because the element $c$ in $\mc{B}_e$ does not correspond to any element of $\mc{A}$.
\end{proof}

We have proved that if $\mc{A} \models \theta_e$ then $\mc{B}_e \models \theta_e$ and $\mc{A} \ncong \mc{B}_e$. Thus $\theta_e$ cannot be a Scott sentence for $\mc{A}$. It follows that, since the $\theta_e$ list all computable $\Pi_3$ sentences, $\mc{A}$ cannot have a computable $\Pi_3$ Scott sentence.
\end{proof}

\section{There is no characterization of the structures with a computable $\Pi_2$ Scott sentence}

Given that we know that some computable structures with $\Pi_2$ Scott sentences have no computable $\Pi_2$ Scott sentences (or even computable $\Pi_3$ Scott sentences), we might like to know how to decide whether or not some particular structure with a computable $\Pi_2$ Scott sentence has a computable $\Pi_2$ Scott sentence.

Having a $\Pi_2$ Scott sentence is equivalent to being $\exists$-atomic, which means that it has a Scott family of existential formulas. This is equivalent (see \cite{MonEffAtomicStructures}) to the following statement: for each tuple $\bar{a} \in \mc{A}$ there is an existential sentence $\varphi(\bar{x})$ such that for any universal sentence $\psi(\bar{x})$ with $\mc{A} \models \psi(\bar{x})$, $\mc{A} \models \forall \bar{x} \; \varphi(\bar{x}) \rightarrow \psi(\bar{x})$. Thus it is not too difficult to tell whether a structure has a $\Pi_2$ Scott sentence: $\Pi^0_4$ if one counts the quantifiers (though we leave it as an open question to determine if this is best possible).

Since it is not too hard to decide whether a computable structure $\mc{A}$ has a $\Pi_2$ Scott sentence, we might as well assume that our given structure $\mc{A}$ does. One possible attempt to characterize whether $\mc{A}$ has a computable $\Pi_2$ Scott sentence is to ask whether it is effectively $\exists$-atomic, that is, whether it has a c.e.\ Scott family of existential formulas. Alvir, Knight, and McCoy \cite{AlvirKnightMcCoy} proved that if a computable structure $\mc{A}$ has a c.e.\ Scott family of existential formulas, then it has a computable $\Pi_2$ Scott sentence. However the converse fails (and while our Theorem \ref{thm:complete} gives a counterexample it is not too hard to construct a one directly).

In this section we will prove that there is no characterization of the structures with a computable $\Pi_2$ Scott sentence. We will use the technique of index set complexity which has been used to great effect in computable structure theory, as suggested in \cite{GoncharovKnight}, e.g., to show that there is no characterization of when a computable structure has a decidable or automatic presentation \cite{HTdec,BHKMN}. Given a listing of all (partial) computable structures in a given langauge, we consider the set of (indices for) structures with the particular property we are interested in. In this case, we consider the set
\[ \{i \mid \text{$\mc{A}_i$ has a computable $\Pi_2$ Scott sentence}\}.\]
This set is naively $\Pi^1_1$, because $\mc{A}_i$ has a computable $\Pi_2$ Scott sentence if and only if (a) $\mc{A}$ has a $\Pi_2$ Scott sentence and (b) there is a computable $\Pi_2$ sentence $\varphi$ such that for all countable structures $\mc{B}$, $\mc{B} \models \varphi$ if and only if $\mc{A} \equiv_2 \mc{B}$. Note that give (a), if $\mc{A} \equiv_2 \mc{B}$ then $\mc{A} \cong \mc{B}$. While (a) is $\Pi^0_4$, (b) is $\Pi^1_1$. If there was a better characterization of when a structure had a computable Scott sentence, then this index set would be simpler than $\Pi^1_1$, e.g., hyperarithmetic. However we show that the set is $\Pi^1_1$-$m$-complete, and hence there is no simpler characterization.

\complete*

\begin{proof}
    By general results on the universality of certain languages up to effective bi-interpretability, we can construct examples in whatever language we wish to. We explained above why the class is $\Pi^1_1$, and it remains to show that it is $\Pi^1_1$-$m$-complete. 
    Let $T \subseteq \omega^{< \omega}$ be a tree. We will define a computable structure $\mc{A} = \mc{A}_T$ which is $\exists$-atomic and hence has a $\Pi_2$ Scott sentence, and such that $\mc{A}$ has a computable $\Pi_2$ Scott sentence if and only if $T$ is well-founded.

    We list all of the computable $\Pi_2$ sentences as $(\theta_e)_{e \in \omega}$ where
    \[ \theta_e = \bigdoublewedge_{i \in \omega} \forall \bar{x}_{e,i} \varphi_{e,i}(\bar{x}_{e,i}) \]
    where the $\varphi_{e,i}(\bar{x}_{e,i})$ are computable $\Sigma_1$ formulas uniformly in $e,i$ and the arities of each of the $\bar{x}_{e,i}$ are also computable in the indices.

    $\mc{A}$ will be $\exists$-atomic and hence have a $\Pi_2$ Scott sentence. It will consist of a number of elements each of which is given various c.e.\ labels. As before, we can take $\mc{A}$ to be the ``bouquet graph'' $\mc{G}_1(\mc{F})$ of a collection $\mc{F}$ of subsets of $\omega$. We introduce three sets of labels. First, we have \textit{sort labels} $(u_e)_{e \in \omega}$ such that only exactly one label holds of each element, dividing the domain into the disjoint sets $U_e = \{ x \in A : u_e(x) \}$; we think of these as different sorts of the structure, and call the elements of $U_e$ the \textit{$e$th sort}. Then we will have two set of labels $(\ell_\sigma)_{\sigma \in \omega^{< \omega}}$ and $(\ell^\dagger_\sigma)_{\sigma \in \omega^{< \omega}}$ we will just call \textit{labels}.

    Within $U_e$ we will diagonalize against $\theta_e$, though this diagonalization may only be successful if $T$ has an infinite path. Given a path $\pi$ through $T$, we will diagonalize by constructing another countable structure $\mc{B}_e = \mc{B}_{e,\pi}$ such that if $\mc{A} \models \theta_e$ then $\mc{B}_e \models \theta_e$ and $\mc{A} \ncong \mc{B}_e$. At each stage $s$ we will have an approximation $\mc{A}_s$ to $\mc{A}$, with $\mc{A} = \bigcup_s \mc{A}_s$. $\mc{B}_e$ will differ from $\mc{A}$ only on the $e$th sort $U_e$. $\mc{B}_e$ will also be built by approximations $\mc{B}_e =\bigcup_s \mc{B}_{e,s}$ with each $\mc{B}_{e,s} \cong \mc{A}_s$, though the construction of $\mc{B}_e$ is non-computable as it requires knowing a path through $T$.

    When we describe the construction of $\mc{A} = \bigcup \mc{A}_s$, we will describe the construction of the $e$th sort. The constructions for the different sorts should be thought of as happening simultaneously.

    During the construction certain stages will be \textit{$e$-expansionary stages} where we get evidence that $\mc{A} \models \theta_e$. We use the variable $k[s]$ to keep track of the number of expansionary stages. The elements of $\mc{A}$ will all be of the form $a_\sigma$ for some $\sigma \in T$. We write $T_s$ for the set of $\sigma$ such that $a_\sigma \in \mc{A}_s$, i.e., for the set of $\sigma$ which correspond to elements of $\mc{A}_s$ at stage $s$. $T_s$ will be a subtree of $T$, with the property that if one child of $\sigma \in T_s$ is in $T_s$, then all children of $\sigma$ are in $T_s$.
    
    \medskip

    \noindent \textit{Construction of the $e$th sort of $\mc{A}$.}

    \medskip

    \noindent \textit{Stage 0.} We begin with $\mc{A}_0$ consisting of a single element $a_\varnothing$ with the single label $\ell_\varnothing$. Thus $T_0 = \{\varnothing\}$. Begin with $k = 0$ as we have not yet had any expansionary stages.

    \medskip

    \noindent \textit{Stage $s+1$.} Suppose that we have constructed $\mc{A}_s$ with $k = k[s]$. We say that a tuple $\bar{x}$ is $k$-small it consists of elements $a_\sigma \in \mc{A}$ of the $e$th sort with $\sigma \in \{0,\ldots,k\}^{\leq k}$ and elements not of the $e$th sort but among the first $k$ elements of $\mc{A}$. First, we check whether this is an expansionary stage\footnote{The conditions to have an expansionary stage, and exactly which elements we should add to $\mc{A}_{s+1}$ at an expansionary stage, are quite subtle. We must ensure that we can perform the construction of $\mc{B}_e$ given below and prove Lemmas \ref{lem:bunion} and Lemma \ref{lem:bsat}.}: Check whether, for each $i \leq k$ and each $k$-small tuple $\bar{x}_i$ we have
    \[ \mc{A}_{s} \models \varphi_i(\bar{x}_i) \]
    where the existential quantifier in $\varphi_i(\bar{x}_i)$ is witnessed by one of the first $s$ witnesses (being careful to use an appropriately dovetailed listing of possible witnesses).
    If this is the case, then stage $s+1$ is an expansionary stage, and set $k[s+1] = k[s]+1$. Otherwise, set $k[s+1] = k[s] + 1$.

    If stage $s+1$ is not an expansionary stage, set $\mc{A}_{s+1} = \mc{A}_s$. If stage $s+1$ is an expansionary stage:
    \begin{enumerate}
        \item for each $\sigma \in T_s \cap \{0,\ldots,k\}^{\leq k}$ with a child $\tau$ in $T \cap \{0,\ldots,k\}^{\leq k}$, put the label $\ell^\dagger_\sigma$ on $a_\sigma$.
        \item for each $\sigma \in T_s \cap \{0,\ldots,k\}^{\leq k}$ which is a leaf of $T_s$ and has a child $\tau$ in $T \cap \{0,\ldots,k\}^{\leq k}$, and each child $\tau \in T$ of $\sigma$, add a new element $a_\tau$ with the labels $\ell_\rho$ for $\rho \preceq \tau$. Let $T_{s+1}$ be $T_s$ together with all of these new $\tau$.
    \end{enumerate}

\medskip

\noindent \textit{End construction.}

\medskip

What $\mc{A}$ looks like depends on whether there are finitely many or infinitely many expansionary stages.
\begin{enumerate}
    \item If there are finitely many expansionary stages, say $k$, then $\mc{A}$ has finitely many elements $a_\sigma$ with each $|\sigma| \leq k$. Write $T_\infty$ for the set of such $\sigma$; $T_\infty = \bigcup T_s$ is a subtree of $T$. Each $a_\sigma$ has the labels $\ell_\rho$ for $\rho \preceq \sigma$; and if $\sigma$ is not a leaf of $T_\infty$, then $a_\sigma$ also has the label $\ell^\dagger_\sigma$.
    \item If there are infinitely many expansionary stages, then $\mc{A}$ has elements $a_\sigma$ for $\sigma \in T$, and each $a_\sigma$ has the labels $\ell_\rho$ for $\rho \preceq \sigma$ and the label $\ell^\dagger_\sigma$. (We have $T_\infty = \bigcup T_s = T$.)
\end{enumerate}
Thus it is easy to see that $\mc{A}$ has a $\Pi_2$ Scott sentence.

\begin{lemma}\label{lem:pi2}
    $\mc{A}$ has a $\Pi_2$ Scott sentence.
\end{lemma}
\begin{proof}
    For each $a_\sigma$, either $a_\sigma$ has the label $\ell^\dagger_\sigma$ and no other element has this label, or $\sigma$ is a leaf of $T_\infty$ and $a_\sigma$ is the unique element with the label $\ell_\sigma$.
\end{proof}

\begin{lemma}\label{lem:wf}
    Suppose that $T$ is well-founded. Then $\mc{A}$ has a computable $\Pi_2$ Scott sentence.
\end{lemma}
\begin{proof}
    Note that $T_\infty = \bigcup T_s$ is a c.e.\ set. Moreover, the set of non-leaves of $T_\infty$ is also a c.e.\ set. Consider the computable $\Pi_2$ sentence which says that:
    \begin{enumerate}
        \item if $x$ has the label $\ell_\sigma$ then it also has all of the labels $\ell_\rho$ for $\rho \prec \sigma$.
        \item if $x$ has the label $\ell^\dagger_\sigma$ then it has the label $\ell_\sigma$ and does not have any label $\ell_\tau$ for $\tau \succ \sigma$.
        \item no $x$ has labels $\ell_{\sigma_1}$ and $\ell_{\sigma_2}$ for incompatible $\sigma_1$ and $\sigma_2$.
        \item if $x$ has the label $\ell_\sigma$ then $\sigma \in T_\infty$.
        \item every $x$ has the label $\ell_\varnothing$ and if $x$ has the label $\ell_\sigma$ and $\sigma$ is not a leaf of $T_\infty$ then $x$ has either the label $\ell^\dagger_\sigma$ or some label $\ell_\tau$ where $\tau$ is a child of $\sigma$ on $T$.
        \item if $x$ has the label $\ell^\dagger_\sigma$, then $\sigma$ is a non-leaf of $T_\infty$.
        \item for each $\sigma \in T_\infty$, there is some $x$ with the label $\ell_\sigma$.
        \item for each $\sigma \in T_\infty$ which is not a leaf, there is some $x$ with the label $\ell^\dagger_\sigma$.
        \item for every distinct $x$ and $y$, either there are some incompatible $\sigma,\tau$ such that one element has the label $\ell_\sigma$ and the other has the label $\ell_\tau$, or there is some $\sigma$ and a child $\sigma^*$ of $\sigma$ such that one element has the label $\ell^\dagger_\sigma$ and the other has the label $\ell_{\sigma^*}$.
    \end{enumerate}
    Since $T$ is well-founded, $T_\infty$ must also be well-founded. We can check that this sentence is true of $\mc{A}$.

    Suppose that a structure $\mc{C}$ satisfies this sentence. For each non-leaf $\sigma$ of $T_\infty$, by (8) there is an element $c_\sigma$ with the label $\ell^\dagger_\sigma$ (and thus, by (1) and (2) also the labels $\ell_\rho$ for $\rho \preceq \sigma$, and by (2) and (3) these are the only labels). By (9) and (2) no other element of $\mc{C}$ has the label $\ell^\dagger_\sigma$.


    For each leaf $\sigma$ of $T_\infty$ by (7) there is an element $c_\sigma$ with the label $\ell_\sigma$. By (1) it also has all of the labels $\ell_\rho$ for $\rho \prec \sigma$, and by (2), (3), and (6) it has no other labels. Thus we have shown that $\mc{A}$ embeds as a substructure of $\mc{C}$.

    Suppose that $d$ is some other element of $\mc{C}$, not included among the $c_\sigma$ above. By (5) $d$ has the label $\ell_\varnothing$. Let $\sigma_0 = \varnothing$. We argue inductively as follows. Given $\sigma_i$ such that $d$ has the label $\ell_{\sigma_i}$, by (5) either (a) $\sigma_i$ is a leaf of $T_\infty$, (b) $d$ has the label $k_{\sigma_i}$, or (c) $d$ has some label $\ell_{\sigma_{i+1}}$ for some child $\sigma_{i+1} \in T_\infty$ of $\sigma_i$. Since $T_\infty$ is well-founded, we can repeat this process to build a sequence of children $\sigma_0 = \varnothing, \sigma_1,\ldots,\sigma_m$ in $T_\infty$ with either (a) $\sigma_m$ is a leaf of $T_\infty$, or (b) $d$ has the label $k_{\sigma_m}$. Recall that there is already an element $c_{\sigma_m}$ with the label $k_{\sigma_m}$, and by (2) and (9) there are no other such elements. Thus $\sigma_m$ is a leaf of $T_\infty$ and $d$ has the label $\ell_{\sigma_m}$. But we already have an element $c_{\sigma_m}$ with the label $\ell_{\sigma_m}$, and by (9) we cannot have two such elements. So no such $d$ exists, that is, $\mc{C} \cong \mc{A}$.
\end{proof}

\begin{lemma}\label{lem:infexp}
    Suppose that $\mc{A} \models \theta_e$. Then there are infinitely many expansionary stages.
\end{lemma}
\begin{proof}
    Let $s_1 < s_2 < s_3 < \cdots < s_k$ be expansionary stages. We must argue that there will later be another expansionary stage, the $k+1$st. For each $i \leq k$ and $k$-small tuple $\bar{x}_i$ we have
    \[ \mc{A} \models \varphi_i(\bar{x}_i).\]
    Since $\varphi_i$ is existential, and $\mc{A} = \bigcup_s \mc{A}_s$, there must be some stage $s_{k+1} > s_k$ such that for each such $i$ and $\bar{x}_i$ we have
    \[ \mc{A}_{s_{k+1}} \models \varphi_i(\bar{x}_i)\]
    with the existential quantifier in $\varphi_i$ witnessed at this stage.
    This stage $s_{k+1}$ is expansionary.
\end{proof}

Suppose that $T$ has an infinite path $\pi$. Let $\mc{B}_e$ be $\mc{A}$ together with another element $c$ of the $e$th sort satisfying $\ell_{\rho}$ for each $\rho \prec \pi$. Clearly $\mc{B}_e \ncong \mc{A}$. However we also want to know that if $\mc{A} \models \theta_e$, then $\mc{B}_e \models \theta_e$. To see this, we will show that $\mc{B}_e$ is also a union $\mc{B}_e = \bigcup \mc{B}_{e,s}$ and that $\mc{B}_{e,s} \cong \mc{A}_s$. Thus one can think of $\mc{A}$ and $\mc{B}_e$ as direct limits of direct systems of the same structures but with different embeddings
\[ \mc{A}_0 \to \mc{A}_1 \to \mc{A}_2 \to \cdots .\]
However the definition of $\mc{B}_{e,s}$ depends on the path $\pi$ and thus is non-computable.

At stage $s$, let $\pi_s$ be the longest initial segment of $\pi$ which on $T_s$. Then $\pi_s$ is a leaf of $T_s$---this is because $T_s$ has the property that if there is any child of $\pi_s$ in $T_s$, then all children of $\pi_s$ from $T$ are in $T_s$. $\mc{B}_{e,s}$ will have domain consisting of elements $\{b_\sigma \mid \sigma \in T_s, \sigma \neq \pi_s\} \cup \{ c \}$. We put the same labels on $b_\sigma$ as on $a_\sigma$, and the same labels on $c$ as on $a_{\pi_s}$. Note that since $\pi_s$ is on the infinite path $\pi$, there cannot be a label $k_{\pi_s}$ on $c$ and $a_{\pi_s}$. We must check that $\mc{B}_e$ is in fact the union of these $\mc{B}_{e,s}$.

\begin{lemma}\label{lem:bunion}
     If there are infinitely many expansionary stages, then $\lim_s \pi_s = \pi$, and so $\mc{B}_{n} = \bigcup_s \mc{B}_{e,s}$.
\end{lemma}
\begin{proof}
    Suppose not, so that for some $t$ for all $s \geq t$ we have $\pi_s = \pi_t$. But then there is $k$ such that $\pi \restriction_{|\pi_{t}|+1}  \in \{0,\ldots,k\}^{\leq k}$, and at some stage $s \geq t$ there is a $k'$th expansionary stage for some $k' \geq k$. At this stage, we add the children of $\pi_t$ on $T$ to $T_{s+1}$, so that $\pi_{t+1}$ is one of these children and strictly extends $\pi_t$.
\end{proof}

Now we can show that $\mc{B}_e \models \theta_e$.

\begin{lemma}\label{lem:bsat}
    Suppose that $T$ has an infinite path. If there are infinitely many expansionary stages, then $\mc{B}_e \models \theta_e$.
\end{lemma}
\begin{proof}
    Fix $i$ and $\bar{y}_i \in \mc{B}_e$. Fix some stage $s$ sufficiently large such that (a) by stage $s$ there have been $k$ expansionary stages, (b) for each $b_\sigma$ in $\bar{y}_i$ of the $e$th sort, $\sigma \in \{0,\ldots,k\}^{\leq k}$ and $|\sigma| < |\pi_s|$, and (c) each element of $\bar{y}_i$ not in the $e$th sort is among the first $k$ elements of $\mc{A}$ (recalling that $\mc{B}_e$ is the same as $\mc{A}$ on these other sorts). Let $\bar{x}_i$ be $\bar{y}_i$ except that each $b_\sigma$ in $\bar{y}_i$ is replaced by $a_\sigma$, and $c$ is replaced by $a_{\pi_s}$, and elements not of the $e$ sort are kept the same. Then at some stage $s' \geq s$ there is for the first time 
 a $k'$th expansionary stage with $\pi_s \in \{0, \ldots, k' \}^{\leq k'}$. Since $\bar{x}_i$ is $k$-small we have $\mc{A}_s \models \varphi_i(\bar{x}_i)$. Since $\mc{B}_{e,s} \cong \mc{A}_{s}$ via an isomorphism mapping $b_\sigma \mapsto a_\sigma$ and $c \mapsto a_{\pi_s}$ (and hence $\bar{y}_i \mapsto \bar{x}_i$), we have  $\mc{B}_{e,s} \models \varphi_i(\bar{y}_i)$. Since $\varphi_i$ is $\Sigma_1$, $\mc{B}_e \models \varphi_i(\bar{y}_i)$. Thus, since this is true for all $i$ and $\bar{y}_i \in \mc{B}_e$, $\mc{B}_e \models \theta_e$.
\end{proof}

We have shown in Lemma \ref{lem:pi2} that $\mc{A}$ has a $\Pi_2$ Scott sentence. If $T$ is well-founded, then Lemma \ref{lem:wf} says that $\mc{A}$ has a computable $\Pi_2$ Scott sentence. Otherwise, if $T$ has an infinite path, we argue that $\mc{A}$ has no computable $\Pi_2$ Scott sentence. If it did, say $\theta_e$, then we have $\mc{A} \models \theta_e$. Then by Lemma \ref{lem:infexp} there are infinitely many $e$-expansionary stages. We construct the structure $\mc{B}_{e,s} \ncong \mc{A}$ using a path through $T$, and by Lemma \ref{lem:bsat} we have that $\mc{B}_e \models \theta_e$. Thus in fact $\theta_e$ cannot have been a Scott sentence for $\mc{A}$. Thus if $T$ has an infinite path, we argue that $\mc{A}$ has no computable $\Pi_2$ Scott sentence.
\end{proof}

\section{Corollaries and other commentary}\label{sec:cor}

\subsection{Complexity of Scott families}

In \cite{AlvirKnightMcCoy} it was shown that if a computable structure has a c.e.\ Scott family of computable $\Sigma_\alpha$ formulas then the structure has a computable $\Pi_{\alpha + 1}$ Scott sentence. (Note that it is easy to construct uncountably many structures with a c.e.\ Scott family of $\Sigma_1$ formulas, so it is necessary here that the structure is computable.) On the other hand, in the same paper it was shown that if a structure has a computable $\Pi_{\alpha + 1}$ Scott sentence then it has a Scott family of computable $\Sigma_\alpha$ formulas, but as the following theorem shows the Scott family is not necessarily c.e.

\begin{corollary}
    There is a computable structure with a computable $\Pi_2$ Scott sentence but with no c.e.\ Scott family of computable $\Sigma_1$ formulas.
\end{corollary}
\begin{proof}
    Any computable structure with a c.e.\ Scott family of computable $\Sigma_1$ formulas has a computable $\Pi_2$ Scott sentence. However the index set of computable structures with a c.e.\ Scott family of computable $\Sigma_1$ formulas is $\Sigma^0_5$ 
    while the index set of computable structures with a computable $\Pi_2$ Scott sentence is $\Pi^1_1$-$m$-complete. Thus there must be a computable structure with a computable $\Pi_2$ Scott sentence but with no c.e.\ Scott family of computable $\Sigma_1$ formulas.
\end{proof}

Despite this, one can ask how bad the Scott family for such a structure must be. Note that by Theorem \ref{thm:pi3} these give necessary but not sufficient conditions to have a computable $\Pi_2$ Scott sentence.

\begin{proposition}
\label{cebound}
    Suppose $A$ is computable and has a computable $\Pi_{\alpha + 1}$ Scott sentence. Then $A$ has a c.e.\ Scott family of computable $\Pi_{\alpha + 1}$ formulas. 
\end{proposition}

\begin{proof}
We will describe a uniform procedure that given $\bar{a} \in A$ produces a computable $\Pi_{\alpha + 1}$ formula defining the automorphism orbit of $\bar{a}$. We know that each of these automorphism orbits is definable by a computable $\Sigma_\alpha$ formula, and so it is enough to find a formula that supports the computable $\Pi_{\alpha}$ type of $\bar{a}$. It is $\Pi_{\alpha}$ to list all the computable $\Pi_{\alpha}$ formulas true of $\bar{a}$. Taking the conjunction of all such formulas we get a formula that supports the type and thus defines the automorphism orbit of $\bar{a}$. This is a conjunction of a $\Pi_{\alpha}$ set of $\Pi_{\alpha}$ formulas, which is also a conjunction of a $\Sigma_{\alpha + 1}$ set of $\Pi_{\alpha + 1}$ formulas, and so is equivalent to a computable $\Pi_{\alpha + 1}$ formula (see Proposition 7.12 of \cite{AshKnight00}). Given $\bar{a}$, we can find (an index for) the $\Pi_\alpha$ set of indices for the $\Pi_\alpha$ formulas we want to take the conjunction of, and so checking that Proposition 7.12 of \cite{AshKnight00} is uniform we can uniformly in $\bar{a}$ find an index for the computable $\Pi_{\alpha+1}$ formula equivalent to the conjunction. Thus we get a c.e.\ Scott family of computable $\Pi_{\alpha+1}$ formulas.
\end{proof}
\color{black}

\begin{proposition}
\label{sabound}
    Suppose $A$ is a computable structure with a computable $\Pi_{\alpha+1}$ Scott sentence. Then $A$ has a $\Sigma^0_{\alpha+2}$ Scott family of computable $\Sigma_\alpha$ formulas. 
\end{proposition}


\begin{proof}
    Since $A$ has a computable $\Pi_{\alpha + 1}$ Scott sentence, it has a Scott family of computable $\Sigma_\alpha$ formulas, and so for each tuple $\bar{a}$ there is a computable $\Sigma_{\alpha}$ formula defining the automorphism orbit of $\bar{a}$. Now consider the set of all pairs $(\bar{a},\psi)$ where $\bar{a}$ is a tuple from $\mc{A}$ and $\psi$ is a $\Sigma_{\alpha}$ formula such that
    \begin{enumerate}
        \item $\mc{A} \models \psi(\bar{a})$;
        \item for all $\bar{b}$, if $\mc{A} \models \psi(\bar{b})$ then for all computable $\Sigma_\alpha$ formulas $\theta(\bar{x})$, if $\mc{A} \models \theta(\bar{a})$ then $\mc{A} \models \theta(\bar{b})$.
    \end{enumerate}
    (1) is $\Sigma^0_{\alpha}$ while (2) is $\Pi^0_{\alpha+1}$. Thus this set is $\Pi^0_{\alpha+1}$. The projection of this set onto the second coordinate is a $\Sigma^0_{\alpha+2}$ Scott family.
\end{proof}

\begin{question}
    Let $A$ be a computable structure with a computable $\Pi_{\alpha + 1}$ Scott sentence. Are the bounds on the Scott family obtained in \Cref{cebound} and \Cref{sabound} best possible? 
\end{question}

\subsection{Pseudo-Scott sentences}

Recall that a pseudo-Scott sentence for a computable structure $\mc{A}$ is a computable sentence $\varphi$ such that for all \textit{computable} structure $\mc{B}$,
\[ \mc{B} \models \varphi \;\Longleftrightarrow\; \mc{A} \cong \mc{B}.\]
There are several examples of structures with a computable pseudo-Scott sentence of a certain complexity but no computable Scott sentence of that complexity \cite{HoDescribingGroups,QuinnEqStructures}. As a corollary we get a much stronger example.

\begin{corollary}
    There is a computable structure $\mc{A}$ with a $\Pi_2$ Scott sentence but no computable $\Pi_2$ Scott sentence, but with a computable $\Pi_2$ sentence $\varphi$ such that, for all hyperarithmetic $\mc{B}$,
    \[ \mc{B} \models \varphi \;\Longleftrightarrow\; \mc{A} \cong \mc{B}.\]
\end{corollary}
\begin{proof}
    Consider the index set of computable structures $\mc{A}$ with a $\Pi_2$ Scott sentence and such that there is a computable $\Pi^0_2$ sentence $\varphi$ such that for all hyperarithmetic structures $\mc{B}$, $\mc{B} \models \varphi$ if and only if $\mc{A} \equiv_2 \mc{B}$ (so that $\mc{A} \cong \mc{B}$). Since we are quantifying universally over hyperarthimetic structures, this index set is $\Sigma^1_1$. Moreover, it is a superset of the $\Pi^1_1$-$m$-complete set of computable structures with a computable $\Pi_2$ sentence. Thus it must be a proper superset, proving the corollary.
\end{proof}

\subsection{Working in $\Mod(\mc{L})$}

Sometimes by working in $Mod(\mc{L})$ rather than with index sets we get a stronger theorem. In this case, it depends on how we relativize the statement and whether we work with boldface or lightface Borel classes.

\begin{proposition}
    The set of structures with a computable $\Pi_2$ Scott sentence is a Borel---in fact (boldface) $\mathbf{\Sigma}^0_3$---set in $Mod(\mc{L})$. It is (lightface) $\Pi^1_1$ but, in a sufficiently rich language, not (lightface) $\Sigma^1_1$.
\end{proposition}
\begin{proof}
    A similar argument as above (in the first paragraph of Theorem \ref{thm:complete}) shows that it is (lightface) $\Pi^1_1$. It is also (boldface) $\mathbf{\Sigma}^1_1$: listing out all of the countably many structures with a computable $\Pi_2$ Scott sentence, together with their Scott sentences, a structure has a computable $\Pi_2$ Scott sentence if and only if it is isomorphic to one of these. This is (boldface) $\mathbf{\Sigma}^0_3$: given $\mc{A}$, we ask whether there exists a structure in this list such that $\mc{A}$ satisfies the corresponding $\Pi^0_2$ Scott sentence. If this set was (lightface) $\Sigma^1_1$, then the index set of computable such structures would be $\Sigma^1_1$; but we know that this index set is $\Pi^1_1$-$m$-complete.
\end{proof}

\begin{corollary}
    The set of structures $\mc{A}$ with an $\mc{A}$-computable $\Pi_2$ Scott sentence is $\Pi^1_1$-Wadge-complete set in $Mod(\mc{L})$.
\end{corollary}
\begin{proof}
    The proof is the same as that of Theorem \ref{thm:complete}, working relative to the tree $T$.
\end{proof}

\subsection{Jump inversion}

We can use the method of jump inversion or Marker extensions to generalize our results from $\Pi_2$ sentences to $\Pi_n$ sentences for arbitrary $n$. We get the following corollaries of our main theorem:

\begin{corollary}\label{cor:nopi3}
    For each computable ordinal $\alpha$ there is a computable structure with a $\Pi_{\alpha+2}$ Scott sentence but with no computable $\Sigma_{\alpha +4}$ Scott sentence.
\end{corollary}

\begin{corollary}\label{cor:complete}
    With $(\mc{A}_i)_{i \in \omega}$ an effective list of (possibly partial) structures in a sufficiently rich language, for each computable ordinal $\alpha$, the set
    \[ \{i \mid \text{$\mc{A}_i$ has a computable $\Pi_{\alpha+2}$ Scott sentence}\}\]
    is $\Pi^1_1$-$m$-complete.
\end{corollary}

Note that Marker extensions are only additive, so that we do not answer whether every computable structure with a $\Pi_n$ Scott sentence has a computable $\Sigma_{2n}$ Scott sentence for $n \geq 3$.

The particular flavour of jump inversion that we use is due to Goncharov, Harizanov, Knight, McCoy, R. Miller, and Solomon \cite{GHKMMS} and also given in Chapter X.3 of \cite{MBook}. In Chapter X.3 of \cite{MBook} Montalb\'an shows that given a computable ordinal $\alpha$ and a structure $\mc{A}$ there is a structure $\Phi_\alpha(\mc{A})$ such that $\mc{A}$ is effectively bi-interpretable with the $\alpha$-canonical structural jump of the image. There is a uniform effective construction from $\mc{A}$ of $\Phi_\alpha(\mc{A})$ as discussed in \cite{ChenGonzalezHT}. We leave to \cite{MBook} the definitions and properties of effective bi-interpretations and jumps of structures (see also \cite{MonJump}).

If two structures are effectively bi-interpretable, then one has a $\Pi_\beta$ Scott sentence if and only if the other does. Thus $\mc{A}$ has a $\Pi_\beta$ Scott sentence if and only if the $\alpha$-canonical structural jump of $\Phi_\alpha(\mc{A})$ does. And the $\alpha$-canonical structural jump of a structure has a $\Pi_\beta$ Scott sentence if and only if the original structure has a $\Pi_{\alpha+\beta}$ Scott sentence. The same is true for computable Scott sentences. Thus:

\begin{proposition}
Let $\mc{A}$ be a computable structure and let $\alpha$ and $\beta$ be computable ordinals.
\begin{itemize}
    \item $\mc{A}$ has a $\Pi_\beta$ Scott sentence if and only if $\Phi_\alpha(\mc{A})$ has a $\Pi_{\alpha+\beta}$ Scott sentence.
    \item $\mc{A}$ has a computable $\Pi_\beta$ Scott sentence if and only if $\Phi_\alpha(\mc{A})$ has a computable $\Pi_{\alpha+\beta}$ Scott sentence.
\end{itemize}
The same is true for $\Sigma$ Scott sentences.
\end{proposition}

From this we can prove the corollaries above.

\begin{proof}[Proof of Corollary \ref{cor:nopi3}]
    Let $\mc{A}$ be a computable structure with a $\Pi_2$ Scott sentence but no computable $\Sigma_4$ Scott sentence. Then $\Phi_{\alpha}(\mc{A})$ is a computable structure with a $\Pi_{\alpha + 2}$ Scott sentence but no computable $\Sigma_{\alpha+4}$ Scott sentence.
\end{proof}

\begin{proof}[Proof of Corollary \ref{cor:complete}]
    Given a $\Pi^1_1$ set $X$, we can construct for $i \in X$ a computable structure $\mc{A}_i$ such that $i \in X$ if and only if $\mc{A}_i$ has a computable $\Pi_2$ Scott sentence. Then $i \in X$ if and only if $\Phi_{\alpha}(\mc{A}_i)$ has a computable $\Pi_{\alpha+2}$ Scott sentence.
\end{proof}

\bibliography{References}

@article {KLM,
    AUTHOR = {Knight, Julia F. and Lange, Karen and McCoy CSC, Charles},
     TITLE = {Computable $\Prod_2$ Scott Sentences},
   JOURNAL = {Preprint}
}

@article {QuinnEqStructures,
    AUTHOR = {Quinn, Sara B.},
     TITLE = {Scott sentences for equivalence structures},
   JOURNAL = {Arch. Math. Logic},
  FJOURNAL = {Archive for Mathematical Logic},
    VOLUME = {59},
      YEAR = {2020},
    NUMBER = {3-4},
     PAGES = {453--460},
      ISSN = {0933-5846,1432-0665},
   MRCLASS = {03D45 (03C57)},
  MRNUMBER = {4081070},
MRREVIEWER = {Nikolay\ Bazhenov},
       DOI = {10.1007/s00153-019-00701-x},
       URL = {https://doi.org/10.1007/s00153-019-00701-x},
}

@article {HoDescribingGroups,
    AUTHOR = {Ho, Meng-Che},
     TITLE = {Describing groups},
   JOURNAL = {Proc. Amer. Math. Soc.},
  FJOURNAL = {Proceedings of the American Mathematical Society},
    VOLUME = {145},
      YEAR = {2017},
    NUMBER = {5},
     PAGES = {2223--2239},
      ISSN = {0002-9939,1088-6826},
   MRCLASS = {03D45 (03C57 20F10)},
  MRNUMBER = {3611333},
MRREVIEWER = {Alexandra\ Andreeva\ Soskova},
       DOI = {10.1090/proc/13458},
       URL = {https://doi.org/10.1090/proc/13458},
}

@incollection {MonEffAtomicStructures,
    AUTHOR = {Montalb\'an, Antonio},
     TITLE = {Effectively existentially-atomic structures},
 BOOKTITLE = {Computability and complexity},
    SERIES = {Lecture Notes in Comput. Sci.},
    VOLUME = {10010},
     PAGES = {221--237},
 PUBLISHER = {Springer, Cham},
      YEAR = {2017},
      ISBN = {978-3-319-50062-1; 978-3-319-50061-4},
   MRCLASS = {03C57 (03D45)},
  MRNUMBER = {3629724},
MRREVIEWER = {Alexandra\ Andreeva\ Soskova},
       DOI = {10.1007/978-3-319-50062-1},
       URL = {https://doi.org/10.1007/978-3-319-50062-1},
}

@article{GHKMMS,
	author = {Goncharov, Sergey and Harizanov, Valentina and Knight, Julia and McCoy, Charles and Miller, Russell and Solomon, Reed},
	coden = {APALD7},
	doi = {10.1016/j.apal.2005.02.001},
	fjournal = {Annals of Pure and Applied Logic},
	issn = {0168-0072},
	journal = {Ann. Pure Appl. Logic},
	mrclass = {03D45 (03C57)},
	mrnumber = {2169684 (2006f:03071)},
	mrreviewer = {Denis R. Hirschfeldt},
	number = {3},
	pages = {219--246},
	title = {Enumerations in computable structure theory},
	url = {http://dx.doi.org/10.1016/j.apal.2005.02.001},
	volume = {136},
	year = {2005}}

@incollection{Sco65,
	address = {Amsterdam},
	author = {Scott, Dana},
	booktitle = {Theory of {M}odels ({P}roc. 1963 {I}nternat. {S}ympos. {B}erkeley)},
	mrclass = {02.35},
	mrnumber = {0200133 (34 \#32)},
	mrreviewer = {E. Engeler},
	pages = {329--341},
	publisher = {North-Holland},
	title = {Logic with denumerably long formulas and finite strings of quantifiers},
	year = {1965}}

@unpublished{MBook,
	author = {Montalb\'an, Antonio},
	label = {Mon24},
	note = {In preparation},
	title = {Computable structure theory: Beyond the arithmetic}}

@article{AGNHTT,
	author = {Alvir, Rachael and Greenberg, Noam and Harrison-Trainor, Matthew and Turetsky, Dan},
	doi = {10.1017/jsl.2021.4},
	fjournal = {The Journal of Symbolic Logic},
	issn = {0022-4812,1943-5886},
	journal = {J. Symb. Log.},
	mrclass = {03D45 (03C57 03C70)},
	mrnumber = {4362932},
	mrreviewer = {Wesley\ Calvert},
	number = {4},
	pages = {1706--1720},
	title = {Scott complexity of countable structures},
	url = {https://doi.org/10.1017/jsl.2021.4},
	volume = {86},
	year = {2021}}

@article{MonSR,
	author = {Montalb{\'a}n, A.},
	doi = {10.1090/proc/12669},
	fjournal = {Proceedings of the American Mathematical Society},
	issn = {0002-9939},
	journal = {Proc. Amer. Math. Soc.},
	mrclass = {03D45 (03C57 03C75 03D60)},
	mrnumber = {3411157},
	mrreviewer = {Rodney G. Downey},
	number = {12},
	pages = {5427--5436},
	title = {A robuster Scott rank},
	url = {scottRank.pdf},
	volume = {143},
	year = {2015}}

@book{Montalbán_2021, 
place={Cambridge}, 
series={Perspectives in Logic}, 
title={Computable Structure Theory: Within the Arithmetic}, 
publisher={Cambridge University Press}, 
author={Montalbán, Antonio}, 
year={2021}, 
collection={Perspectives in Logic}}

@InProceedings{MonJump,
author="Montalb{\'a}n, Antonio",
editor="Ambos-Spies, Klaus
and L{\"o}we, Benedikt
and Merkle, Wolfgang",
title="Notes on the Jump of a Structure",
booktitle="Mathematical Theory and Computational Practice",
year="2009",
publisher="Springer Berlin Heidelberg",
address="Berlin, Heidelberg",
pages="372--378",
isbn="978-3-642-03073-4"
}

@article {BHKMN,
	AUTHOR = {Bazhenov, Nikolay and Harrison-Trainor, Matthew and
	Kalimullin, Iskander and Melnikov, Alexander and Ng, Keng
	Meng},
	TITLE = {Automatic and polynomial-time algebraic structures},
	JOURNAL = {Journal of Symbolic Logic},
	FJOURNAL = {The Journal of Symbolic Logic},
	VOLUME = {84},
	YEAR = {2019},
	NUMBER = {4},
	PAGES = {1630--1669},
	ISSN = {0022-4812},
	MRCLASS = {03C57 (03D05 03D20 03D45 03D80 68Q45)},
	MRNUMBER = {4045992},
	MRREVIEWER = {Andrey Morozov},
	DOI = {10.1017/jsl.2019.26},
	URL = {https://doi-org.helicon.vuw.ac.nz/10.1017/jsl.2019.26},
}

@article{GoncharovKnight,
  title={Computable Structure and Non-Structure Theorems},
  author={Sergey Goncharov and Julia A. Knight},
  journal={Algebra and Logic},
  year={2002},
  volume={41},
  pages={351-373},
  url={https://api.semanticscholar.org/CorpusID:116943513}
}

@article{HTdec,
author = {Harrison-Trainor, Matthew},
title = {There is no classification of the decidably presentable structures},
journal = {Journal of Mathematical Logic},
volume = {18},
number = {02},
pages = {1850010},
year = {2018},
doi = {10.1142/S0219061318500101},

URL = { 
    
        https://doi.org/10.1142/S0219061318500101
    
    

},
eprint = { 
    
        https://doi.org/10.1142/S0219061318500101
    
    

}
}

@unpublished{ChenGonzalezHT,
    author = {Ruiyuan Chen and David Gonzalez and Matthew Harrison-Trainor},
    title = {Optimal Syntactic Definitions of Back-and-Forth Types},
    note = {preprint}
}

@article {Nadel,
	AUTHOR = {Nadel, Mark},
	TITLE = {Scott sentences and admissible sets},
	JOURNAL = {Ann. Math. Logic},
	FJOURNAL = {Annals of Pure and Applied Logic},
	VOLUME = {7},
	YEAR = {1974},
	PAGES = {267--294},
	ISSN = {0168-0072},
	MRCLASS = {02H10 (02F27)},
	MRNUMBER = {0384471 (52 \#5348)},
	MRREVIEWER = {Nigel J. Cutland},
}

@article {Makkai,
	AUTHOR = {Makkai, M.},
	TITLE = {An example concerning {S}cott heights},
	JOURNAL = {J. Symbolic Logic},
	FJOURNAL = {The Journal of Symbolic Logic},
	VOLUME = {46},
	YEAR = {1981},
	NUMBER = {2},
	PAGES = {301--318},
	ISSN = {0022-4812},
	MRCLASS = {03C70},
	MRNUMBER = {613284},
	MRREVIEWER = {Bienvenido F. Nebres},
	DOI = {10.2307/2273623},
	URL = {https://doi.org/10.2307/2273623},
}

@book {AshKnight00,
	AUTHOR = {Ash, C. J. and Knight, J.},
	TITLE = {Computable structures and the hyperarithmetical hierarchy},
	SERIES = {Studies in Logic and the Foundations of Mathematics},
	VOLUME = {144},
	PUBLISHER = {North-Holland Publishing Co., Amsterdam},
	YEAR = {2000},
	PAGES = {xvi+346},
	ISBN = {0-444-50072-3},
	MRCLASS = {03D45 (03-02 03D28 03D30)},
	MRNUMBER = {1767842},
	MRREVIEWER = {Rodney G. Downey},
}

@article {KnightMillar,
	AUTHOR = {Knight, J. F. and Millar, J.},
	TITLE = {Computable structures of rank {$\omega^{\rm CK}_1$}},
	JOURNAL = {J. Math. Log.},
	FJOURNAL = {Journal of Mathematical Logic},
	VOLUME = {10},
	YEAR = {2010},
	NUMBER = {1-2},
	PAGES = {31--43},
	ISSN = {0219-0613},
	MRCLASS = {03D45 (03C57)},
	MRNUMBER = {2802081},
	MRREVIEWER = {Wesley Calvert},
	DOI = {10.1142/S0219061310000912},
	URL = {https://doi.org/10.1142/S0219061310000912},
}

@article {Harrison68,
	AUTHOR = {Harrison, Joseph},
	TITLE = {Recursive pseudo-well-orderings},
	JOURNAL = {Trans. Amer. Math. Soc.},
	FJOURNAL = {Transactions of the American Mathematical Society},
	VOLUME = {131},
	YEAR = {1968},
	PAGES = {526--543},
	ISSN = {0002-9947},
	MRCLASS = {02.77},
	MRNUMBER = {0244049 (39 \#5366)},
	MRREVIEWER = {A. S. Ferebee},
}

@article {Morley70,
	AUTHOR = {Morley, Michael},
	TITLE = {The number of countable models},
	JOURNAL = {J. Symbolic Logic},
	FJOURNAL = {The Journal of Symbolic Logic},
	VOLUME = {35},
	YEAR = {1970},
	PAGES = {14--18},
	ISSN = {0022-4812},
	MRCLASS = {02.50},
	MRNUMBER = {288015},
	MRREVIEWER = {F. R. Drake},
	DOI = {10.2307/2271150},
	URL = {https://doi-org.helicon.vuw.ac.nz/10.2307/2271150},
}

@article{HTSpectra,
title = {Scott ranks of models of a theory},
journal = {Advances in Mathematics},
volume = {330},
pages = {109-147},
year = {2018},
issn = {0001-8708},
doi = {https://doi.org/10.1016/j.aim.2018.03.012},
url = {https://www.sciencedirect.com/science/article/pii/S0001870818300999},
author = {Matthew Harrison-Trainor}
}

@article {AlvirKnightMcCoy,
	AUTHOR = {Alvir, Rachael and Knight, Julia F. and McCoy CSC, Charles},
	TITLE = {Complexity of {S}cott sentences},
	JOURNAL = {Fund. Math.},
	FJOURNAL = {Fundamenta Mathematicae},
	VOLUME = {251},
	YEAR = {2020},
	NUMBER = {2},
	PAGES = {109--129},
	ISSN = {0016-2736},
	MRCLASS = {03C75 (03C57)},
	MRNUMBER = {4125858},
	DOI = {10.4064/fm865-6-2020},
	URL = {https://doi-org.helicon.vuw.ac.nz/10.4064/fm865-6-2020},
}

@article {HTIgusaKnight18,
	AUTHOR = {Harrison-Trainor, Matthew and Igusa, Gregory and Knight, Julia
	F.},
	TITLE = {Some new computable structures of high rank},
	JOURNAL = {Proc. Amer. Math. Soc.},
	FJOURNAL = {Proceedings of the American Mathematical Society},
	VOLUME = {146},
	YEAR = {2018},
	NUMBER = {7},
	PAGES = {3097--3109},
	ISSN = {0002-9939},
	MRCLASS = {03D45 (03C57)},
	MRNUMBER = {3787370},
	MRREVIEWER = {Nikolay Bazhenov},
	DOI = {10.1090/proc/13967},
	URL = {https://doi-org.helicon.vuw.ac.nz/10.1090/proc/13967},
}
\bibliographystyle{alpha} 

\end{document}